\documentclass[12pt, A4]{article}
\usepackage{amsmath,amssymb,amsfonts,mathrsfs,hyperref,microtype, amsthm}
\usepackage{bm, enumerate}
\usepackage{authblk}

\usepackage{geometry}
\usepackage{eufrak}
\usepackage{graphicx,color}
\usepackage[T1]{fontenc}
\usepackage{rotating}
\usepackage{booktabs}
\usepackage{mathtools}
\usepackage{float}
\usepackage{breqn}
\usepackage{graphicx}
\usepackage{caption}

\newtheorem{thm}{Theorem}[section]
\newtheorem{rem}{Remark}[section]
\newtheorem{definition}{Definition}[section]
\newtheorem{lem}{Lemma}[section]
\newtheorem{prop}{Proposition}[section]
\newtheorem{cor}{Corollary}[section]

\newtheorem{ex}{Example}[section]

\numberwithin{equation}{section}

\def\HH{ \EuFrak H}
\def\N{{\rm I\kern-0.16em N}}
\def\R{{\rm I\kern-0.16em R}}
\def \E{\mathbb{E}}
\def\P{{\rm I\kern-0.16em P}}
\def\F{{\rm I\kern-0.16em F}}
\def\B{{\rm I\kern-0.16em B}}
\def\C{{\rm I\kern-0.46em C}}
\def\G{{\rm I\kern-0.50em G}}

\newcommand{\iu}{{i\mkern1mu}}

\font\eka=cmex10
\usepackage{ae}

\def\ind{\mathrel{\hbox{\rlap{%
\hbox to 7.5pt{\hrulefill}}\raise6.6pt\hbox{\eka\char'167}}}}
\begin{document}

\title{{\Large \bf  Stein characterizations for linear combinations of
    gamma random variables}}
\author[1]{Benjamin Arras}
\author[2]{Ehsan Azmoodeh }
\author[3]{Guillaume Poly}
\author[4]{Yvik Swan }
\affil[1]{{\it Laboratoire Jacques-Louis Lions, Universit\'e Pierre et Marie Curie, France}}
\affil[2]{{\it Faculty of Mathematics, Ruhr University Bochum, Germany}}
\affil[3]{{\it Institut de Recherche Math\'ematiques, Universit\'e de Rennes 1, France}}
\affil[4]{{\it Mathematics department, Universit\'{e} de Li\`ege, Belgium}}
\renewcommand\Authands{ and }

\date{ }

\maketitle

\abstract In this paper we propose a new, simple and explicit
mechanism allowing to derive Stein operators for random variables
whose characteristic function satisfies a simple ODE.  We apply this
to study random variables which can be represented as linear
combinations of (non necessarily independent) gamma distributed random
variables. The connection with Malliavin calculus for random variables
in the second Wiener chaos is detailed. An application to McKay Type I
random variables is also outlined.


\tableofcontents

\section{Introduction and overview} 
\subsection{On Stein's method}

Stein's method is a popular and versatile probabilistic toolkit for
stochastic approximation.  Presented originally in the context of
Gaussian CLTs with dependent summands (see \cite{S72}) it has now been
extended to cater for a wide variety of quantitative asymptotic
results, see \cite{Chen-book} for a thorough overview in the context
of Gaussian approximation
or~https://sites.google.com/site/steinsmethod for an up-to-date list
of references on non-Gaussian and non-Poisson Stein-type results. 

Given two random objects $F, F_{\infty}$, Stein's method allows to
compute fine bounds on quantities of the form
\begin{equation*} \sup_{h\in \mathcal{H}} \left|
    \E\left[ h(F) \right] - \E \left[ h(F_{\infty}) \right] \right|
\end{equation*}
with $\mathcal{H}$ some meaningful class of functions with respect to
which both $F$ and $F_{\infty}$ are integrable (Zolotarev's integral
probability metrics \cite{Z83}, which include e.g.\ the total
variation distance and the Kolmogorov distance, are of the above
form).  The method rests on three pins:
\begin{enumerate}[A.]
\item a ``Stein pair'', i.e.\ a linear operator and a class of
  functions $(\mathcal{A}_\infty, \mathcal{F}(\mathcal{A}_\infty))$
  such that $ \E \left[ \mathcal{A}_\infty(f(F_{\infty})) \right] = 0$
  for all test functions $f \in \mathcal{F}(\mathcal{A}_\infty)$;
\item a ``Stein equation and its magic factors'', i.e.\ a contractive
  inverse operator $\mathcal{A}_\infty^{-1}$ acting on the centered
  functions $\bar{h} = h - \E h(F_{\infty})$ in $\mathcal{H}$ and
  tight bounds on $\mathcal{A}_\infty^{-1}(\bar{h})$ and its
  derivatives;
\item handles on the structure of $F$ (such as $F=F_n=T(X_1, \ldots,
  X_n)$ a $U$-statistic, $F = F(X)$ a functional of an isonormal
  Gaussian process, $F$ a statistic on a random graph, etc.). 
\end{enumerate}
Given the conjunction of these three elements one can then apply some
form of transfer principle:
\begin{equation}\label{eq:6}
 \sup_{h\in \mathcal{H}} \left|
    \E\left[ h(F) \right] - \E \left[ h(F_{\infty}) \right] \right| 
   = \sup_{h\in \mathcal{H}} \left|
    \E\left[ \mathcal{A}_\infty\left( \mathcal{A}_\infty^{-1}(\bar h(F)) \right) \right]\right|;
\end{equation}
remarkably the right-hand-side of the above is often much more
amenable to computations than the left-hand-side, even in particularly
unfavourable circumstances. This has resulted in Stein's method
delivering several striking successes (see
\cite{B-H-J,Chen-book,n-pe-1}) which have led the method to becoming
the recognised and acclaimed tool it is today.

Given a target $F_{\infty}$, the identification of an appropriate
Stein operator $\mathcal{A}_{\infty}$ is the cornerstone of Stein's
method.  While historically most practical implementations relied on
adhoc arguments, several general tools exist, including Stein's
\emph{density approach} \cite{S86} and Barbour's \emph{generator
  approach} \cite{B90}.  A general theory for Stein operators is
available in \cite{LRS}. It is easy to see that, given any
sufficiently regular target $F_{\infty}$, there are infinitely many
admissible choices of operator $\mathcal{A}_{\infty}$ and the
difficulty is to identify those that shall lead to quantities useful
for tackling \eqref{eq:6}. In many important cases, particularly
Pearson or Ord random variables, these ``useful'' operators are first
order differential operators (see \cite{dobler}) or difference
operators (see \cite{LS}). Higher order differential operators are
sometimes necessary to characterize more complex distributions, see
\cite{g-variance-gamma,p-r-r} for random variables with densities
satisfying second order differential equations and \cite{g-2normal,
  g-2normal2, g-m-s} for random variables which can be written as the
product of independent Pearson variables satisfying certain
conditions.\\
\\
The purpose of this paper is to add to the literature on Stein's
method by proposing a new, simple and explicit mechanism allowing to
derive Stein operators for random variables whose characteristic
function satisfies a simple ODE.  We apply this to study random
variables which can be represented as linear combinations of (non
necessarily independent) gamma distributed random variables. The
connection with Malliavin calculus for random variables in the second
Wiener chaos is detailed. An application to the study of McKay Type I
random variables is also outlined.

\subsection{The Malliavin-Stein method  and its  extensions}\label{sec:mall-based-appr}
If $F_{\infty}$ is standard Gaussian random variable then the operator
is $ \mathcal{A}_{\infty}f(x) = f'(x) - xf(x)$ with
$\mathcal{F}(\mathcal{A}_{\infty})$ the class of all differentiable
functions such that $\mathbb{E} \left| f'(F_{\infty})
\right|<\infty$. The simple structure of both the operator and the
class, as well as the wide variety of possible choices for $F$, entail
that all stars align beautifully well for a Gaussian target and that
many paths are open for exploration.  A particularly fruitful path was
opened by Ivan Nourdin and Giovanni Peccati who, in \cite{n-pe-ptrf},
identified the possibility of intertwining Stein's method with
Malliavin calculus.  Given a sufficiently regular centered random
variable $F$ with finite variance and smooth density, the first step
in this direction is to define its Stein kernel $\tau_F(F)$ through
the integration by parts formula
\begin{equation}\label{eq:9}
  \E[\tau_F(F) f'(F) ] = \E \left[ F f(F)
\right] \mbox{ for all absolutely continuous } f,
\end{equation}
(see Stein's monograph \cite{S86} for the origins of this concept and
for a detailed study when $F$ is Pearson distributed).  Then, for
$f_h$ a solution to $f_h'(x) - xf_h(x) = h(x) -\E[h(F_{\infty})]$
(i.e.\ $f_h = \mathcal{A}_{\infty}^{-1}(\bar{h})$), we
can write
\begin{align*}
  \E[h(F)] - \E[h(F_{\infty})] &  = \E \left[ f_h'(F)
                                        - F f_h(F) \right]  = \E
                                 \left[ (1-\tau_F(F)) f_h'(F)
                                 \right]. 
\end{align*}
By Cauchy-Schwarz inequality we have
\begin{align*}
  \left|  \E[h(F)] - \E[h(F_{\infty})]  \right| \le \| f_h'\|
  \sqrt{\E \left[ (1-\tau_F(F))^2  \right]}
\end{align*} 
and at this stage two good things happen: (i) the constant
$\sup_{h \in \mathcal{H}}\|f_h'\|$ (which is intrinsically Gaussian
and does not depend on the law of $F$) is bounded for wide and
relevant classes $\mathcal{H}$; (ii) the quantity
\begin{equation}\label{eq:7}
  S(F \, || \, F_{\infty}) = \E \left[ (1-\tau_F(F))^2 \right]
\end{equation}
(called the \emph{Stein discrepancy}) is tractable, via Malliavin
calculus, as soon as $F$ is a sufficiently regular functional of a
Gaussian process because,  in this case,  the Stein kernel is 
$\tau_F(F)   =\langle
DF,-DL^{-1}F\rangle_{\HH}
$, where $D$ and $L^{-1}$ stand for Malliavin derivative and pseudo-inverse Ornstein-Uhlenbeck operators.  These two realizations spawned an entire new field of research known
 as ``Malliavin-Stein method'' or as ``Nourdin-Peccati'' method, see
 \cite{n-pe-ptrf, n-pe-1} or the dedicated webpage
 https://sites.google.com/site/malliavinstein.

Extensions of the Malliavin-Stein method outside of the Gaussian
framework have been studied as well. The first natural target to
tackle is $F_{\infty} = 2 G - d, d>0$ where $G$ has Gamma law with
parameter $d/2$ (i.e.\ $F_{\infty}$ is centered Gamma) with operator
\begin{equation}\label{eq:13}
  \mathcal{A}_{\infty} f(x) = 2(x+d) f'(x) - xf(x),
\end{equation} 
see \cite{n-pe-2,a-m-m-p}. Mimicking the Gaussian approach outlined above, one
captures the difference in law between $F_{\infty}$ and some arbitrary
$F$ by considering solutions to the ODE $ 2(x+d) f_h'(x) - xf_h(x) =
h(x) -\E[h(F_{\infty})]$ and 
\begin{align*}
   \left|  \E[h(F)] - \E[h(F_{\infty})]  \right| &  =  \mathbb{E} \left[
  2(F+d) f_h'(F) - Ff_h(F) \right] \le \| f_h'\| \mathbb{E}\left|  2(F+d) - \tau_F(F) \right|.
\end{align*}
Again it is necessary, for the last bound to be of interest, that
$f_h'$ be bounded and that $\tau_F(F)$ have good properties; see
\cite[Section 3.3]{n-pe-ptrf} for an illustration as well as
\cite{kusuoka,viquez} for further explorations for targets
$F_{\infty}$ belonging to the Pearson family. 

Important progress in this direction is due to Robert Gaunt
\cite{g-2normal,g-variance-gamma}. In \cite{g-2normal} he shows that
if $F_\infty {=} N_1 \times N_2$ where $N_1$ and $N_2$ are two
independent $\mathscr{N} (0,1)$ random variables then its law is
characterized by a \emph{second order} Stein equation
\begin{equation}
  \label{eq:14}
  x f''(x) + f'(x) - x f(x) = h(x) - \E [h (F_\infty)]
\end{equation}
and in \cite{g-variance-gamma} he studies the entire family of Variance Gamma
distributions (see Example~\ref{ex:vraga} below), obtains Stein
operators $\mathcal{A}_{\infty}$ and also bounds on the solutions on
the resulting Stein equations $\mathcal{A}_{\infty}f = h$ under
smoothness assumptions on $h$.  These results are used in
\cite{thale}, where Gaunt's estimates are combined with higher order Stein kernels firstly introduced in \cite{NP10}
(see below, Definition \ref{Def : Gamma})
in order to extend the scope of the Nourdin-Peccati approach to
targets of the form
\begin{equation}\label{eq:target-2wiener}
 F_\infty = \sum_{i=1}^{d} \alpha_{\infty,i} (N^2_i -1)
\end{equation}
where the coefficients $\{ \alpha_{\infty,i} \, : \, i=1,\cdots,d\}$
are non-zero and distinct and the $N_i$ are i.i.d.\ standard Gaussian
(actually $d=2$ in \cite{thale}, but we shall consider the general
case from here onwards).

As we shall see (e.g.\ in \eqref{eq:SME}), random variables of the
form \eqref{eq:target-2wiener} are characterized by Stein operators
$\mathcal{A}_{\infty}$ which are differential operators of order
$d$. In order for Nourdin and Peccati's approach to function for such
operators one needs to introduce higher order versions of the Stein
kernel~\eqref{eq:9}, one for each degree of the operator. This is
exactly the purpose of Section \ref{sec:stein-oper-second}. \begin{definition}[see \cite{n-pe-1}]\label{Def : Gamma}  Let $F\in
  \mathbb{D}^{\infty}$ the class of infinitely many times Malliavin
  differentiable random variables (see \cite[Chapter 2]{n-pe-1} for a
  detailed discussion). The sequence of random variables
  $\{\Gamma_i(F)\}_{i\geq 0}\subset \mathbb{D}^\infty$ is recursively
  defined as follows. Set $\Gamma_0(F) = F$ and, for every $i\geq
  1$,
  \[\Gamma_{i}(F) = \langle DF,-DL^{-1}\Gamma_{i-1}(F)\rangle_{\HH}.
\]
\end{definition}
\noindent Iterated Gammas from Definition \ref{Def : Gamma} are higher
order versions of the Stein kernel \eqref{eq:9}; by definition we have
$\Gamma_1(F) = \tau_F(F)$.  Also note how
$\mathbb{E} \left[ \tau_F(F) \right] = \mbox{Var}(F)$ and (see again
\cite{n-pe-1}) the %
cumulants of the random element $F$ and the iterated Malliavin
$\Gamma$- operators are linked by the relation
$\kappa_{r+1}(F)=r! \E [\Gamma_r(F)] \mbox{ for } r=0,1,\cdots.$

Targets $F_{\infty}$ of the form \eqref{eq:target-2wiener} admit (see
below) operators $\mathcal{A}_{\infty}f = \sum_{j=0}^d a_j(x) f^{(j)}$
with $a_j$ polynomials and $d \ge 1$. Mimicking the Gaussian and Gamma
cases, a direct extension of the Nourdin-Peccati approach then
consists, in principle, in writing out
\begin{align*}
  \mathbb{E}[h(F)] - \mathbb{E}[h(F_{\infty})] & = \mathbb{E} \left[
                                                 \sum_{j=0}^d a_j(F)
                                                 f_h^{(j)}(F) \right] 
   = \mathbb{E} \left[ \sum_{j=0}^d (\tilde{a}_j(F) - \Gamma_j(F)) f_h^{(j)}(F)
    \right].
\end{align*}
for $f_h$ a solutions to the ODE
$\mathcal{A}_{\infty}f(x) = h(x) - \mathbb{E}[h(F_{\infty})]$. In
order for this approach to be useful it is necessary that both
$\Gamma_j(F)$ and $f_h^{(j)}$ be tractable. So far the question of
finding tight bounds on solutions to such higher order Stein equations
is open; this seems to be a difficult problem to tackle in all
generality.

Estimates on the derivatives of solutions to Stein equations are,
however, not crucial for a version of Stein's method to apply to
variables of the form \eqref{eq:target-2wiener}, see paragraph after Proposition \ref{p:static}, and in more details \cite{aaps17}.

\section{Stein-type characterization and main results}
\subsection{Stein operators for the second Wiener chaos}
\label{sec:stein-oper-second}
The aim of this section is to use the recent findings in \cite{a-p-p} to derive "appropriate" stein equation, i.e. differential operators of finite order with polynomial coefficients, for random elements in the second Wiener chaos. Indeed, following
\cite{n-po-1,a-p-p}, first we define two crucial polynomials $P$ and $Q$ as
follows:
\begin{equation}\label{polynomialP}
 Q(x)=\big( P(x)\big)^{2}=\Big(x \prod_{i=1}^{d}(x - \alpha_{\infty, i} ) \Big)^{2}.
\end{equation}
Next, for random element $F$ living in a finite sum of Wiener chaoses, we consider the following quantity
(whose first occurrence is in \cite{a-p-p})
\begin{equation}\label{eq:Delta}
\Delta(F, F_{\infty}):= \sum_{r=2}^{\text{deg}(Q)} \frac{Q^{(r)}(0)}{r!} \frac{\kappa_{r}(F)}{2^{r-1}(r-1)!}.
\end{equation}
Then the following result holds: 
\begin{prop}\cite[Proposition 3.2]{a-p-p}\label{p:static}
Let $F$ be a centered random variable living in a finite sum of Wiener chaoses. Moreover, assume that
\begin{itemize}
 \item[\bf (i)]$\kappa_r (F) = \kappa_r (F_\infty)$, for all $2 \le r \le d+1=\text{deg}(P)$, and
\item[\bf (ii)] 
\begin{equation*}
              \E \Bigg[ \sum_{r=1}^{d+1} \frac{P^{(r)}(0)}{r! \
                2^{r-1}} \Big( \Gamma_{r-1}(F) -   \E[\Gamma_{r-1}(F)]
              \Big) \Bigg]^2 = 0. 
  \end{equation*}
\end{itemize}
Then, $F \stackrel{{\rm law}}{=} F_\infty,$ and $F$ belongs to the second Wiener chaos.
\end{prop}



\noindent As we shall see in Section \ref{sec:stein-oper-second},
Proposition \ref{p:static} leads to Stein operators for random
variables belonging to the Second Wiener chaos. By analogy with the
case of a Gaussian target, it appears that the quantity
$\Delta(F, F_{\infty})$ is the second-chaos equivalent of the
(first-Wiener chaos) Stein discrepancy $S(F || F_{\infty})$
\eqref{eq:7}. Moreover we have shown, in a separate publication
\cite{aaps17}, that estimating this quantity directly (without
requiring any bounds on solutions to Stein equations) leads to bounds
on the 2-Wasserstein distance between the law of $F$ and the law of
$F_{\infty}$.

We now show how item {\bf (ii)} of Proposition \ref{p:static} can be
used to derive a Stein operator for $F_\infty$. To this end, set
\begin{align*}
  &  a_l= \frac{P^{(l)}(0)}{l! 2^{l-1}}, \quad 1 \le l \le d+1,\\
  &  b_l= \sum_{r=l}^{d+1} a_r \E[\Gamma_{r-l+1}(F_\infty)] =
    \sum_{r=l}^{d+1} \frac{a_r}{(r-l+1)!} \kappa_{r-l+2}(F_\infty),
    \quad 2 \le l \le d+1. 
\end{align*}
Now, we introduce the following differential operator of order $d$
(acting on functions $f \in C^d(\R)$) : 
\begin{equation}\label{eq:SME} 
\mathcal{A}_\infty f (x):= \sum_{l=2}^{d+1} (b_l - a_{l-1} x )
f^{(d+2-l)}(x) - a_{d+1} x f(x). 
\end{equation}
Then, we have the following result (see Section
\ref{sec:steins-method-second} for a proof).

\begin{thm}[Stein characterization]\label{thm:SMC}
  Assume that $F$ is a general centered random variable living in a
  finite sum of Wiener chaoses (and hence smooth in the sense of
  Malliavin calculus).  Then $F {=} F_\infty$ (equality in
  distribution) if and only if
  $\E \left[ \mathcal{A}_\infty f (F) \right] =0$ for all  polynomials $f:\R \to \R$.  
  \end{thm}

 \begin{ex}{ \rm Consider the special case of only two non-zero distinct eigenvalues $\lambda_1$ and $\lambda_2$, i.e.
\begin{equation}\label{eq:target}
F_\infty=\lambda_1(N^2_1 -1) + \lambda_2 (N^2_2 -1)
\end{equation}
where $N_1, N_2 \sim \mathscr N (0,1)$ are independent. In this case, the polynomial $P$ takes the form $P(x)=x (x - \lambda_1)(x - \lambda_2)$. Simple 
calculations reveal that $P'(0)=\lambda_1 \lambda_2, P''(0)= -2 (\lambda_1 + \lambda_2)$, and $P^{(3)}(0)=3!$. Also, $\kappa_2(F_\infty)=\E [\Gamma_{1}(F_\infty)]=2(\lambda_1^2 + \lambda_2^2)$, and 
$\kappa_3(F_\infty)= 2 \E [\Gamma_{2}(F_\infty)]= 4 (\lambda_1^3 + \lambda_2^3)$. Then, the Stein equation $(\ref{eq:SME})$ reduces to 
\begin{equation}\label{eq:stein-equation-2}
 \mathcal{A}_\infty f (x) = -4 (\lambda_1 \lambda_2 x + (\lambda_1 + \lambda_2) \lambda_1 \lambda_2) f''(x) + 2 \left(\lambda_1^2 + \lambda_2^2+ (\lambda_1 + \lambda_2) x \right) f'(x) - x f(x) 
 \end{equation}
 We also remark that when $\lambda_1=-\lambda_2=\frac12$, and hence
 $F_\infty \stackrel{\text{law}}{=} N_1 \times N_2$, the Stein's
 equation $(\ref{eq:stein-equation-2})$ coincides with that in
 \cite[equation (1.9)]{g-2normal}. One has to note that for general $\lambda_1$ and $\lambda_2$, the random variables of the form $(\ref{eq:target})$ lie outside the {\it Variance-Gamma} class, see also Example \ref{ex:vraga}. 
 }
 \end{ex}

\subsection{A Fourier approach to Stein characterizations}
\label{sec:four-appr-stein}
 
The characteristic functions $\phi_F(\xi) = \mathbb{E}[e^{\iu \xi F}]$
(we drop the indexation in ${\infty}$ and write $F$ instead of
$F_{\infty}$, $\alpha_i$ instead of $\alpha_{\infty, i}$ etc.\ from now on) of random variables of the form
\eqref{eq:target-2wiener} satisfy a simple ODE with polynomial
coefficients, namely
\begin{equation}\label{eq:12}
 {\prod_{j=1}^d (1-2\iu \xi\alpha_j)} \phi_F'(\xi) = {-2\xi  \sum_{j=1}^d
  \alpha_j^2 \prod_{k\neq j}(1-2\iu \xi \alpha_k)} \phi_F(\xi). 
\end{equation}
Such are but particular cases of a wide family of variables to which
the following simple lemma applies (see
Section~\ref{sec:steins-method-second} for a proof).

\begin{lem}\label{sec:four-appr-stein-1}
  Let $(a_k)_{0\le k \le d}$ and $(b_k)_{0 \le k \le d'}$ be real
  numbers and consider the polynomials
  $A_d(\xi) = \sum_{k=0}^d a_k \xi^k$ and
  $B_{d'}(\xi) = \sum_{k=0}^{d'} b_k \xi^k$ with $d, d'\in \N$. Suppose that the random
  variable $F$ has differentiable characteristic function $\phi_F$
  such that
  \begin{equation}
    \label{eq:1}
    A_d(\iu \xi) \phi_F'(\xi) = \iu B_{d'}(\iu \xi)  \phi_F(\xi)
  \end{equation}
  for all $\xi\in \R$.  Let $Y$ be a real valued random variable such
  that $\E[|Y|]<+\infty$. Then $Y \stackrel{\mathrm{law}}{=} F$ if and
  only if
\begin{equation}
  \label{eq:4}
  \mathbb{E} \left[   F   \mathcal{A}_df(F) -
\mathcal{B}_{d'}f(F) \right] = 0
\end{equation}
for all test functions $f \in \mathcal{S}(\R)$ the Schwartz space of
smooth functions with rapid decrease, where
\begin{align}
  \label{eq:8}
   \mathcal{A}_d = \sum_{k=0}^d a_k \frac{d^k}{dx^k}  \mbox{ and }
   \mathcal{B}_{d'} = \sum_{k=0}^{d'} b_{k} \frac{d^k}{dx^k}. 
\end{align}
The differential operator
$f \mapsto \mathcal{A}_{\infty}f(x) = x \mathcal{A}_df(x) -
\mathcal{B}_{d'}f(x)$ is a Stein operator for $F$ with Stein class
$\mathcal{S}(\R)$.
\end{lem}
\noindent From here onwards all test functions $f$ are supposed to belong to
$\mathcal{S}(\R)$.

\begin{ex}\label{ex:simpleexamples}{\rm 
If $F$ is a normal random variable with mean $\mu$ and variance $\sigma^2$ then $\phi_F(\xi) = e^{\iu \mu \xi -\sigma^2\xi^2/2}$ so that
$\phi_F'(\xi) = \iu (\mu+\sigma^2(\iu \xi)) \phi_F(\xi)$ and, in the
notations of Lemma~\ref{sec:four-appr-stein-1}: $d=0$, $a_0 = 1$,
$d'=1$, $b_0 = \mu$, and $b_1 = \sigma^2$ so that 
\begin{equation*}
  \mathbb{E} \left[ F f(F) - \left( \mu f(F)  + \sigma^2 f'(F) \right) \right]=0
\end{equation*}
as expected (see \cite{Chen-book}). 
}
\end{ex}

\begin{ex} \label{ex:vraga} {\rm If $F$ is Variance-Gamma distributed, then
  its cumulant generating function is, in the classical
  parameterization,
\begin{equation*}
  \log \phi_F(\xi) = \mu \iu \xi +2 \lambda \log \gamma - \lambda \log
  \left( \alpha^2 -(\beta+\iu \xi)^2 \right)
\end{equation*}
so that Lemma~\ref{sec:four-appr-stein-1} applies with $d=2$, $a_0 =
\alpha^2-\beta^2$, $a_1 = -2 \beta$,  and $a_2 = -1$, $d'=2$, $b_0 =
\mu (\alpha^2 - \beta^2)+ 2\lambda \beta$, $b_1 = 2(\lambda-\mu \beta)$, and $b_2
= -\mu$ so that 
\begin{align*}
 & \mathbb{E} \left[ F\left((\alpha^2-\beta^2) f(F) -2\beta f'(F) -
      f''(F)\right)\right. \\
&\qquad \left. - \left((\mu (\alpha^2 -\beta^2)+ 2\lambda \beta) f(F)
      +2(\lambda-\mu \beta)f'(F) - \mu f''(F)  \right)\right]=0
\end{align*}
or, after simplifications, 
\begin{equation*}
  \mathbb{E} \left[(F-\mu) f''(F) + \left( 2\beta(F-\mu) + 2\lambda
    \right) f'(F) + \left(  (\alpha^2-\beta^2)(F-\mu)  + 2\lambda \beta \right) f(F)\right]
  =0.  
\end{equation*}
This is the result obtained by \cite[Lemma 3.1]{g-variance-gamma}.
}
\end{ex}
\begin{ex}\label{ex:four-appr-stein-2}{\rm
  Take $\alpha_i = 1$ for all $i\ge1$ in
  \eqref{eq:target-2wiener}, i.e.\  $F= \sum_{i=1}^{d} (N^2_i -1) \sim \chi^2_{(d)}$ is a centered
  chi-squared random variable with $d$ degree of freedom. The CF of a
  chi-squared distributed random variable is
  $\phi(\xi)= (1-2 \iu \xi)^{-d/2}$, and so 
  \begin{equation*}
    (1-2 \iu \xi) \phi_F'(\xi) = - 2 d   \xi \phi_F(\xi).
  \end{equation*}
  Hence, in the notations of Lemma~\ref{sec:four-appr-stein-1}:
 $d=1$, $a_0 = 1, a_1 = -2$, $d'=1$, $b_0 = 0$, and $b_1=2d$ so that
\begin{equation*}
  \mathbb{E} \left[ F f(F)\right]=2\mathbb{E}\left[ (F+d)f'(F) \right].
\end{equation*}
This is the same as \eqref{eq:13}. 
}
\end{ex}

\begin{ex}\label{ex:maccay}{\rm
  Random variable $F$ follows the type I McKay distribution with
  parameters $a>-(1/2), b>0$ and $c>1$ when its PDF is proportional to
  the function
  \begin{equation}
    \label{eq:22}
\forall x \in \R_+,  f_I(x)  =  x^ae^{-xc/b} I_{a}(x/b) 
  \end{equation}
  where $I_a(\cdot)$ denotes the modified Bessel function of the first kind and of order $a$, see \cite{hoal04} for context and further references.  Direct
  computations lead to
\begin{equation}
  \label{eq:25}
(\log  \phi_F)'(\xi) =  -\iu  \frac{(1+2a)bc  -(1+2a)b^2 (\iu
  \xi)}{1-c^2+2 cb (\iu 
  \xi) - b^2(\iu \xi)^2}. 
\end{equation}
Lemma \ref{sec:four-appr-stein-1} applies and we deduce that if $F$ is
type I McKay then 
\begin{equation}
  \label{eq:26}
  \mathbb{E} \left[ \left( (1-c^2)F +(1+2a)bc \right) f(F) +
    \left(2cb F-(1+2a)b^2  \right) f'(F) - b^2F f''(F) \right] = 0 
\end{equation}
for all $f \in \mathcal{S}(\R)$. 
}
\end{ex}

\subsection{Stein operators for sums of independent gamma}
\label{sec:stein-operators-sums}

Before stating the next theorem, we need to introduce some
notations. For any $d$-tuple $(\lambda_1,...,\lambda_d)$ of real
numbers, we define the symmetric elementary polynomial of order
$k\in\{1,...,d\}$ evaluated at $(\lambda_1,...,\lambda_d)$ by:
\begin{align*}
e_{k}(\lambda_1,...,\lambda_d)=\sum_{1\leq i_1<i_2<...<i_k\leq d}\lambda_{i_1}...\lambda_{i_k}.
\end{align*}
We set, by convention, $e_{0}(\lambda_1,...,\lambda_d)=1$. Moreover,
for any $(c_1,...,c_d)\in\mathbb{R}^*$ and any $k\in\{1,...,d\}$, we
denote by $(\lambda c)$ the $d$ tuple
$(\lambda_1c_1, \ldots, \lambda_dc_d)$ and by $(\lambda c)_k$ the
$d-1$ tuple
$(\lambda_1c_1, \ldots, \lambda_{k-1}c_{k-1},\lambda_{k+1}c_{k+1}, \ldots,
\lambda_dc_d)$

The object of interest in this section are the following
generalizations of \eqref{eq:target-2wiener}: for  $d \ge 1$, $(m_1, \ldots, m_d) \in \mathbb{N}^d$
$(\lambda_1, \ldots, \lambda_d) \in \R^{\star}$  all distinct we
consider 
\begin{equation}
  \label{eq:23}
  F = \sum_{i=1}^d\lambda_i (
\gamma_i(m_i\alpha_i, c_i) - m_i\alpha_i c_i)
\end{equation}
where, for any $(\alpha,c)\in\mathbb{R}_+^*$, we denote by
$\gamma(\alpha,c)$ a gamma random variable with parameters $(\alpha,c)$ with
density
\begin{align*}
  \forall x\in\mathbb{R}_+^*,\
  \gamma_{\alpha,c}(x)=\dfrac{1}{c
  \Gamma(\alpha)}\left(\frac{x}{c}\right)^{\alpha-1}e^{-\frac{x}{c}}
\end{align*}
and CF 
\begin{equation*}
  \phi_{\gamma(\alpha, c)}(\xi) = \left( 1-\iu c \xi \right)^{-\alpha}.
\end{equation*}
The family $\{\gamma_i(m_i\alpha_i, c_i) ,\ i=1,...d\}$ is a collection of independent random variables. Applying Lemma \ref{sec:four-appr-stein-1} we obtain the following
(proof in Section \ref{sec:steins-method-second}). 


\begin{thm}\label{benjarras}
  Let $F$ be as in \eqref{eq:23} and let $Y$ be a real valued random
  variable such that $\E[|Y|]<+\infty$. Then
  $Y \stackrel{\mathrm{law}}{=} F$ if and only if
\begin{align}
  &\E \bigg[
    \big(Y+\sum_{i=1}^d\lambda_i{m_i\alpha_i}{c_i}\big)(-1)^d\bigg(\prod_{j=1}^d{\lambda_j}{c_j}\bigg)f^{(d)}(Y)+\sum_{l=1}^{d-1}(-1)^l\bigg(Ye_{l}((\lambda 
    c))\nonumber\\  
  &\quad \quad +\sum_{k=1}^d\lambda_km_k\alpha_kc_k\left(e_{l}((\lambda
    c))-e_{l}((\lambda c)_k)\right)\bigg)f^{(l)}(Y)+Yf(Y)
    \bigg]=0\label{eq:2} 
\end{align}
for all $f \in \mathcal{S}(\R)$. 
\end{thm}

Taking $\alpha_k=1/2$ and $c_k=2$ in the previous theorem implies the following straightforward corollary:

\begin{cor}\label{benjarras2}
Let $d \ge 1$, $q\geq 1$ and $(m_1, \ldots, m_d) \in \mathbb{N}^d$ such that $m_1+...+m_d=q$. Let $(\lambda_1, \ldots, \lambda_d) \in
\R^{\star}$ pairwise distincts and consider: 
\begin{equation*}
   F = \sum_{i=1}^{m_1}\lambda_1(N_i^2-1)+\sum_{i=m_1+1}^{m_1+m_2}\lambda_2(N_i^2-1)+...+\sum_{i=m_1+\ldots+m_{d-1}+1}^q\lambda_d(N_i^2-1),
\end{equation*}
Let $Y$ be a real valued random variable such that $\E[|Y|]<+\infty$. Then $Y \stackrel{\text{law}}{=} F$ if and only if
\begin{align}
  &\E \bigg[ \big(Y+\sum_{i=1}^d\lambda_im_i\big)(-1)^d2^d\bigg(\prod_{j=1}^d\lambda_j\bigg)f^{(d)}(Y)+\sum_{l=1}^{d-1}2^l(-1)^l\bigg(Ye_{l}(\lambda_1,...,\lambda_d)\nonumber\\ &+\sum_{k=1}^d\lambda_km_k\left(e_{l}(\lambda_1,...,\lambda_d)-e_{l}((\underline{\lambda}_k)\right)\bigg)f^{(l)}(Y)+Yf(Y) \bigg]=0,\label{eq:3}
\end{align}
for all $f\in S(\mathbb{R})$.
\end{cor}

\begin{ex}\label{ex:eicheltha}{ \rm
Let $d=1$, $m_1=q\geq 1$ and $\lambda_1=\lambda>0$. The differential operator reduces to (on smooth test function $f$):
\begin{align*}
-2\lambda(x+q\lambda)f'(x)+xf(x).
\end{align*}
This differential operator is similar to the one characterizing the gamma distribution of parameters $(q/2,1/(2\lambda))$. Indeed, we have, for $F\stackrel{\text{law}}{=}\gamma\big(q/2,1/(2\lambda)\big)$, on smooth test function, $f$:
\begin{align*}
\E\bigg[Ff'\big(F\big)+\big(\frac{q}{2}-\frac{F}{2\lambda}\big)f\big(F\big)\bigg]=0
\end{align*}
We can move from the first differential operator to the second one by performing a scaling of parameter $-1/(2\lambda)$ and the change of variable $x=y-q\lambda$.
}
\end{ex}

\begin{ex}{ \rm
Let $d=2,\ q=2$, $\lambda_1=-\lambda_2=1/2$ and $m_1=m_2=1$. The differential operator reduces to (on smooth test function $f$):
\begin{align*}
\mathcal{A}(f)(x)&=4(x+\left\langle m,\lambda \right\rangle)\lambda_1\lambda_2f''(x)-2\big[xe_1(\lambda_1,\lambda_2)+\lambda_1m_1(e_1(\lambda_1,\lambda_2)-e_1(\lambda_2))\\
&+\lambda_2m_2(e_1(\lambda_1,\lambda_2)-e_1(\lambda_1))\big]f'(x)+xf(x),\\
&=-xf''(x)-f'(x)+xf(x),
\end{align*}
where we have used the fact that
$e_1(\lambda_1,\lambda_2)=\lambda_1+\lambda_2=0,\ e_1(\lambda_2)=\lambda_2=-1/2,\ e_1(\lambda_1)=\lambda_1=1/2$. Therefore, up to a minus sign factor,
we retrieve the differential operator associated with the random
variable $F=N_1\times N_2.$
}
\end{ex}

We conclude this section by comparing the Stein-type operators defined
by the Fourier approach with those obtained by the Malliavin calculus
tools in \eqref{eq:SME} (see Section \ref{sec:steins-method-second}
for a proof).

\begin{prop}\label{same}
  The Stein-type operators defined in Corollary \ref{benjarras2} and
  in (\ref{eq:SME}) coincide, up to some normalizing constant.
\end{prop}
\subsection{Stein operators for projections of multivariate gamma }
Independence of the contributions, as required in \eqref{eq:23}, is
not crucial. Indeed, consider all random variables of the form
\begin{equation}
  \label{eq:15}
  F = \left\langle  \Gamma - K, \lambda
  \right\rangle = \sum_{i=1}^d \lambda_i (\Gamma_i -k_i)
\end{equation}
with $K = (k_1, \ldots, k_d)\in \R^d$ and
$\Gamma = (\Gamma_1, \ldots, \Gamma_d)$ a $d$-variate gamma
distributed random variable defined as follows.

\begin{definition}[\cite{KP51}]\label{def:multgamma}
  A random vector $\Gamma = (\Gamma_1, \ldots, \Gamma_d)$ has a
  $d$-variate gamma distribution in the sense of Krishnamoorthy and
  Parthasarathy \cite{KP51} with degree of freedom $\nu = 2\alpha$ and
  covariance matrix $C$ if its characteristic function is 
  \begin{equation}
    \label{eq:16}
    \phi_{\Gamma}(t_1, \ldots, t_d) = \left| I_d - \iu  C T \right|^{-\alpha}
  \end{equation}
  with $t_j \ge 0$ for all $j$, $\left| \cdot \right|$ the determinant
  operator, $I_d$ the $(d \times d)$-identity, $\alpha>0$,
  $T = \mathrm{diag}(t_1, \ldots, t_d)$, and $C$ a symmetric
  positive definite $d \times d$ matrix.
\end{definition}
\noindent Conditions on $C$ and $\alpha$ under which \eqref{eq:16} is
a bona fide characteristic function have been thoroughly adressed in
the literature, see \cite{DVJ97,EK09,royen16} and references therein. In the
sequel we suppose that these conditions are satisfied. 

\begin{lem}\label{lem:sums-dependent-gamma}Let $\Lambda= \mathrm{diag}(\lambda_{1}, \ldots,
  \lambda_d)$ and $C = \left( c_{ij} \right)_{1\le i, j \le
    d}$ a symmetric positive definite matrix. For all $\xi \in \R$ we have 
  \begin{equation}
    \label{eq:20}
    \left| I_d - \iu \xi C \Lambda \right|  = \sum_{j=0}^d  (-1)^j r_j (\iu \xi)^j
  \end{equation}
with $r_0=1$ and 
\begin{equation}
  \label{eq:21}
  r_j =\sum_{\substack{S \subset [d] \\ \#(S) = j}} \left| C
  \right|_S \prod_{j \in S} \lambda_j,
\end{equation}
(the summation in \eqref{eq:21} is over all collections  $S$   of indices in   $[d] = \left\{
  1, \ldots, d \right\}$ with cardinality $\#(S)=j$, and $\left|
  C \right|_S$ is the determinant of the matrix $\left(
  C_{ij}\right)_{i,j\in S}$).
\end{lem}
\begin{ex}\label{ex:stein-oper-d3}
  If $d=3$ and $C = \left( c_{i, j} \right)_{1\le i, j \le 3} $ then
\begin{align*}
  r_0 & = 1 \\
  r_1 & = c_{1} \lambda_1 + c_{2} \lambda_2 + c_{3} \lambda_3 \\
  r_2 & = (c_{1}c_{2} - c_{12}^2)\lambda_1\lambda_2 + (c_{1} c_{3}
        -c_{13}^2)\lambda_1 \lambda_3 + (c_{2}
        c_{3}- c_{23}^2)\lambda_2 \lambda_3 \\
  r_3 & = |C| \lambda_1 \lambda_2 \lambda_3
\end{align*}
(we also write $c_j$ instead of $c_{jj}$ for $j=1, 2, 3$). 
\end{ex}
From Lemma \ref{lem:sums-dependent-gamma} we deduce the CF of linear
combinations of marginals of multivariate gamma random vectors: if
$\Gamma$ has marginals $\Gamma_j \sim \gamma(\alpha, c_j)$ and $F$ is
as in \eqref{eq:15} then, letting $\kappa = \sum_{j=1}^d \lambda_j
k_j$: 
\begin{align*}
  \phi_F(\xi)  & = \mathbb{E} \left[ e^{\iu \xi F} \right]  = e^{-\iu
                 \alpha \xi \sum_{j=1}^d \lambda_j k_j} \mathbb{E} \left[
                 e^{\iu  \sum_{j=1}^d (\xi \lambda_j) \Gamma_j} \right] \\
               & =  e^{-\iu \alpha \kappa \xi } \phi_{\Gamma}(\xi
                 \lambda_1, \ldots, \xi \lambda_d) \\
               & =  e^{-\iu \alpha \kappa \xi} \left( \sum_{j=0}^d (-1)^j r_j (\iu \xi)^j \right)^{-\alpha}
\end{align*}
with $(r_j)_{0\le j \le d}$ given in
Lemma~\ref{lem:sums-dependent-gamma}. Taking derivatives  we obtain 
\begin{align*}
  \left( \sum_{j=0}^d (-1)^j r_j (\iu \xi)^j\right) \phi_F'(\xi) =
  - \iu \alpha 
  \left(  \kappa 
    \sum_{j=0}^d (-1)^j r_j (\iu \xi)^j  +  
    \sum_{j=1}^{d} (-1)^{j} j r_{j} (\iu \xi)^{j-1}\right) \phi_F(\xi). 
\end{align*}
Applying Lemma \ref{sec:four-appr-stein-1} we deduce, after
straightforward simplifications:

\begin{thm}\label{theo:stein-oper-ga}
  Let $F$ be defined in \eqref{eq:15} and $(r_j)_{j=1, \ldots, d}$ as
  in \eqref{eq:21}. Set $r_{d+1} = 0$.  Let $Y$ be a real valued
  random variable such that $\E[|Y|]<+\infty$. Then
  $Y \stackrel{\text{law}}{=} F$ if and only if
\begin{align}
 & \mathbb{E} \left[ (F+\alpha(\kappa-r_1)) f(F) + \sum_{j=1}^{d}  (-1)^j  \left(
  r_j(F +\alpha \kappa) - \alpha(j+1)r_{j+1}) \right)f^{(j)}(F)
 \right]   =0 \label{eq:18}
\end{align}
for all test functions $f \in \mathcal{S}(\R)$.
\end{thm}

\begin{rem}
  If $F$ is of the form \eqref{eq:23} with all shape coefficients
  identical then
  $ F = \sum_{i=1}^d \lambda_i\sum_{j=1}^{m_i}(\gamma_j(\alpha, c_i)-
  \alpha c_i)$. Letting $m= \sum_{j=1}^dm_i$, then $F$ is of the form
  \eqref{eq:15} for $\Gamma$ a $m$-variate gamma random variable with
  $m\times m$ diagonal correlation matrix
  $C = \mathrm{diag}(((c_1)_{m_1}, \ldots, (c_d)_{m_d}))$ (we write
  $(x)_q = (x, \ldots, x)$ a vector of length $q$). Applying Theorem
  \ref{theo:stein-oper-ga} will lead, via \eqref{eq:18}, to an
  operator of order $m>d$ which coincides with \eqref{eq:3} (and thus
  \eqref{eq:SME}) only when $m_i=1$ for all $i$.
\end{rem}

\begin{ex}\label{ex:dis2}
If $d=2$  then 
  $F = \left\langle \Gamma - K ,
    \lambda\right\rangle$ has second-order differential Stein
  operator
\begin{align}\nonumber 
  \mathcal{A}f(x) & = (x+\alpha  (\kappa-r_1))f(x)  - \left\{  r_1x + \alpha 
             \left( r_1 \kappa  -2
           (c_1c_2-c_{12}^2)\lambda_1\lambda_2 \right)\right\} f'(x) \\
  \label{eq:24}
  & \quad + \lambda_1 \lambda_2\big(c_1c_2-c_{12}^2\big)\bigg(x
           + \alpha \kappa  \bigg)f''(x).
\end{align}
(recall that $\kappa = \sum_{j=1}^d \lambda_j k_j$ and $r_1 =
\sum_{j=1}^d \lambda_jc_j$). 
\end{ex}

\subsection{Application: McKay Type I and combinations of two gamma
  variates }\label{sec:applications}

We conclude the paper with applications of the identities in the case
$d=2$. There is interest, even in this simple situation, in obtaining
handles on law of sums and differences of correlated gamma variates as
these have applications in performance analysis, see e.g.\
\cite{hoal04}.  Recall example~\ref{ex:maccay} and the corresponding
operator
\begin{align}\label{eq:29}
\mathcal{A}_{\mathrm{McKay}}f(x) =    \left(x + \frac{(1+2a)bc}{ 1- c^2} \right) f(x) +
    \frac{2cb x-(1+2a)b^2}{1-c^2} f'(x) -\frac{b^2}{1-c^2} x f''(x) 
\end{align}
for type I McKay random variables with parameters $a, b, c$ (see its
pdf defined in \eqref{eq:22}). From \eqref{eq:26} (applied to
functions of the form $f(x) = x^n$, along with a continuity argument
for extending the identity to functions not in $\mathcal{S}(\R)$) we
immediately deduce
\begin{align*}
&  \mathbb{E} \left[ F \right] = \frac{(1+2a)bc}{c^2-1} \\
&  \mathbb{E} [F^2] = \frac{(2a+1)b^2(2(a+1)c^2+1)}{(c^2-1)^2}
\end{align*}
(see \cite[Equation (6)]{hoal04}), as well as the formula 
\begin{equation}
  \label{eq:27}
  (1-c^2) \mathbb{E}[F^{n+1}] + bc(1+2(a+n)) \mathbb{E}[F^n] -
  nb^2(1+2a+n-1) \mathbb{E} [F^{n-1}] = 0
\end{equation}
for all $n\ge 2$. 

\begin{cor}\label{cor:appl-mckay-type}
  McKay Type I random variables can be represented as projections of
  bivariate Gamma random variables with degree of freedom $2\alpha$
  and covariance matrix $C =
  \begin{pmatrix}
    c_1 & c_{12} \\
    c_{12}& c_2
  \end{pmatrix}$ whenever
  \begin{align*}
    &a = \alpha-1/2 \\
    &b = 2\frac{c_1c_2- c_{12}^2}{\sqrt{(c_1+c_2)^2 - 4(c_1c_2- c_{12}^2)}}\\
    & c = \frac{c_1+c_2}{\sqrt{(c_1+c_2)^2 - 4(c_1c_2- c_{12}^2)}}.
  \end{align*}
\end{cor}
\begin{rem}
  Corollary \ref{cor:appl-mckay-type} contains Theorems 3, 4 and 5
  from \cite{hoal04}. In that paper they consider also the so-called
  McKay Type II distribution for which our method also applies; we do
  not perform the computations here.
\end{rem}
\begin{proof}
  Taking $K=0$ and $\lambda_1 = \lambda_2 = 1$ in Example
  \ref{ex:dis2} we obtain that combinations of dependent Gamma random
  variables $G_1 \sim \mathrm{Gamma}(\alpha, c_1)$ and
  $G_2 \sim \mathrm{Gamma}(\alpha, c_2)$ with identical shape
  parameter and covariance $C$ have operator
\begin{align}
  \mathcal{A}_{G_1+G_2}f(x) & = (x-\alpha(c_1+c_2))f(x) - \left(
    (c_1+c_2)x -2\alpha(c_1c_2 - c_{12}^2)
  \right)f'(x)\nonumber \\& \qquad  + (c_1c_2 - c_{12}^2)xf''(x).  \label{eq:28}
\end{align}  We identify the coefficients in \eqref{eq:28} and \eqref{eq:29} to
  get the system of 4 equations: 
  \begin{align*}
&  \frac{bc}{1-c^2}(1+2a)   = -\alpha(c_1+c_2), \quad  2\frac{bc}{1-c^2} = -(c_1+c_2) \\
&  \frac{b^2}{1-c^2}(1+2a)  = - 2 \alpha(c_1c_2-c_{12}^2),
  \quad 
 \frac{b^2}{1-c^2} = -    (c_1c_2-c_{12}^2).
  \end{align*}
  Solving for $a, b, c$ in terms of $\alpha, c_1, c_2$ and $c_{12}$ we
  immediately deduce that $a = \alpha - 1/2$ is necessary, so that the
  system reduces to
\begin{align*}
&  b^2 = (c_1c_2-c_{12}^2) (c^2-1), \quad 2 bc =  (c_1+c_2)(c^2-1)
\end{align*}
and  the result follows.
\end{proof}
\noindent
We end this subsection by discussing infinite divisibility of the law of projections of multivariate gamma distribution. Infinite divisibility of multivariate gamma distribution has been addressed thoroughly in the literature (see \cite{Gr84,Ba89,EK06,EK09}). Thanks to the previous corollary, we are able to explicit the Lévy measure of the sum of two dependent gamma random variables using the parametrization $(a,b,c)$ with $a>-(1/2)$, $b>0$ and $c>1$. We have the following straightforward corollary.

\begin{cor}\label{ID2}
Let $(G_1,G_2)$ be a 2-dimensional gamma random vector of parameters $2\alpha>0$ and covariance matrix $C$ such that $c_1c_2>c_{12}^2$ and $c_1+c_2>1$. Then, the law of $G_1+G_2$ is infinitely divisible and its Lévy-Khintchine formula is given by:
\begin{align}\label{eq:ID2}
\forall t\in\mathbb{R},\ \phi_{G_1+G_2}(t)=\exp\bigg(\int_0^{+\infty}\big(e^{itx}-1\big)(\frac{1}{2}+a)\big(e^{-\frac{c-1}{b}x}+e^{-\frac{c+1}{b}x}\big)\frac{dx}{x}\bigg),
\end{align}
with
\begin{align*}
&a=\alpha-\frac{1}{2},\\
&b=2\frac{c_1c_2- c_{12}^2}{\sqrt{(c_1+c_2)^2 - 4(c_1c_2- c_{12}^2)}},\\
&c=\frac{c_1+c_2}{\sqrt{(c_1+c_2)^2 - 4(c_1c_2- c_{12}^2)}}.\\
\end{align*}
Moreover, we have the following identity in law:
\begin{align}\label{eq:IL}
G_1+G_2=\gamma_1+\gamma_2
\end{align}
where $\gamma_1$ and $\gamma_2$ are independent gamma random variables with parameters $(a+1/2,(c-1)/b)$ and $(a+1/2,(c+1)/b)$ respectively.
\end{cor}

\begin{proof}
Let $a,b$ and $c$ be as in the statement of the corollary. By Corollary \ref{cor:appl-mckay-type}, we know that $G_1+G_2$ has the same law as a McKay type I random variable with parameters $(a,b,c)$. Then, by (\ref{eq:25}), we have:
\begin{align*}
(\log  \phi_{G_1+G_2})'(\xi) =  -i  \frac{(1+2a)bc -(1+2a)b^2 (i
  \xi)}{1-c^2+2 cb (i\xi) - b^2(i \xi)^2}. 
\end{align*}
Performing a partial fraction decomposition, we obtain straightforwardly:
\begin{align*}
(\log  \phi_{G_1+G_2})'(\xi)=ib(\frac{1}{2}+a)\bigg(\dfrac{1}{c-1-ib\xi}+\dfrac{1}{c+1-ib\xi}\bigg).
\end{align*}
Now, we note that:
\begin{align*}
\dfrac{1}{c-1-ib\xi}=\int_0^{+\infty}\exp(-(c-1-ib\xi)x)dx
\end{align*}
and similarly for the other term. By standard computations, we obtain formula (\ref{eq:ID2}). The identity (\ref{eq:IL}) follows trivially.
\end{proof}

\section{Proofs}\label{sec:steins-method-second}

\begin{proof}[Proof of Theorem \ref{thm:SMC}]
 Repeatedly using the Malliavin integration by parts formulae
  \cite[Theorem 2.9.1]{n-pe-1}, we obtain for any $2 \le l \le d+2$
  that
\begin{align}
  \E\left[F f^{(d-l+2)}(F) \right]
  =   \E\left[ f^{(d)}(F)
  \Gamma_{l-2}(F) \right] +  \sum_{r=d-l+3}^{d-1} \E\left[
  f^{(r)}(F) \right] \E\left[ \Gamma_{r+l-d-2}(F) \right].\label{eq:com1}
\end{align}
For indices $l=2,3$, the second term in the right hand side of
$(\ref{eq:com1})$ is understood to be $0$. Summing from $l=2$ up to
$l=d+2$, we obtain that

\begin{equation}\label{com2}
\begin{split}
\sum_{l=2}^{d+2} a_{l-1} \, \E\left[F f^{(d-l+2)}(F) \right]  & = \sum_{l=2}^{d+2} a_{l-1}  \, \E\left[ f^{(d)}(F) \Gamma_{l-2}(F) \right] \\
& \hskip1cm +  \sum_{l=4}^{d+2} a_{l-1}  \, \sum_{r=d-l+3}^{d-1} \E\left[ f^{(r)}(F) \right] \E\left[ \Gamma_{r+l-d-2}(F) \right]\\
& =  \sum_{l=1}^{d+1} a_{l}  \, \E\left[ f^{(d)}(F) \Gamma_{l-1}(F) \right] \\
& \hskip1cm +  \sum_{l=3}^{d+1} a_{l}  \, \sum_{r=d-l+2}^{q-2} \E\left[ f^{(r)}(F) \right] \E\left[ \Gamma_{r+l-d-1}(F) \right]\\
& = \sum_{l=1}^{d+1} a_{l}  \, \E\left[ f^{(d)}(F) \Gamma_{l-1}(F) \right] \\
& \hskip1cm +  \sum_{l=2}^{d+1} a_{l}  \, \sum_{r=1}^{l-2} \E \left[ f^{(d-r)}(F) \right] \E \left[ \Gamma_{l-r-1}(F) \right].
\end{split}
\end{equation}
On the other hand, 
\begin{equation}\label{eq:com4}
\begin{split}
\sum_{l=2}^{d+1} b_l \, \E \left[ f^{(d+2-l)}(F) \right] &=
\sum_{l=0}^{d-1} b_{l+2}  \E \left[ f^{(d-l)} (F) \right]\\
&=   \sum_{l=0}^{d-1}  \left[ \sum_{r=l+2}^{d+1} a_r \E (
  \Gamma_{r-l-1}(F_\infty) ) \right] \E \left[ f^{(d-l)}(F) \right]\\ 
&= \sum_{r=2}^{d+1} a_r \sum_{l=0}^{r-2} \E \left[
  \Gamma_{r-l-1}(F_\infty)  \right] \times \E \left[ f^{(d-l)}(F)
\right]. 
\end{split}
\end{equation}
Wrapping up, we finally  arrive at
\begin{equation}\label{eq:com5}
\begin{split}
\E \left[ \mathcal{A}_\infty f (F) \right] & = - \E \Bigg[  f^{(d)}(F) \times \Big( \sum_{r=1}^{d+1} a_r \left[ \Gamma_{r-1}(F) - \E[\Gamma_{r-1}(F)] \right] \Big) \Bigg] \\
& \hskip 1cm +  \sum_{r=2}^{d+1} a_r \sum_{l=0}^{r-2}  \left\{ \E [
  f^{(d-l)}(F) ] \times \Big(  \E \left[ \Gamma_{r-l-1}(F_\infty)
  \right] - \E \left[ \Gamma_{r-l-1}(F) \right] \Big) \right\}\\ 
& = - \E \Bigg[  f^{(d)}(F) \times \Big( \sum_{r=1}^{d+1} a_r \left[
  \Gamma_{r-1}(F) - \E[\Gamma_{r-1}(F)] \right] \Big) \Bigg] \\ 
&\hskip 1cm +  \sum_{r=2}^{d+1} a_r \sum_{l=0}^{r-2} \frac{ \E [ f^{(d-l)}(F) ]}{(r-l-1)!} \times \Big( \kappa_{r-l}(F_\infty) - \kappa_{r-l}(F) \Big).
\end{split}
\end{equation}
We are now in a position to prove the claim.  First we assume that
$F \stackrel{\text{law}}{=} F_\infty$. Then obviously
$\kappa_{r}(F)=\kappa_{r}(F_\infty)$ for $r=2,\cdots,2d+2$, and moreover, random variable $F$ belongs to the second Wiener
chaos. Hence, according to \cite[Lemma 3]{a-p-p}, the Cauchy--Schwarz inequality, and the hypercontractivity property of the Wiener chaoses \cite[Theorem 2.7.2]{n-pe-1}, we obtain that
\begin{equation*}
\begin{split}
\Big \vert \E \left[ \mathcal{A}_\infty f (F) \right] \Big \vert & \le \sqrt{\E \left[f^{(d)}(F) \right]^2} \times 
\sqrt{ \E \Big[ \sum_{r=1}^{d+1} a_r \left( \Gamma_{r-1}(F) -
    \E[\Gamma_{r-1}(F)] \right) \Big]^2} \\
& =  \sqrt{\E \left[f^{(d)}(F) \right]^2} \times \sqrt{\Delta(F,F_\infty)}\\
&=  \sqrt{\E \left[f^{(d)}(F) \right]^2} \times \sqrt{\Delta(F_\infty,F_\infty)}= 0.
\end{split}
\end{equation*}
Conversely, assume that $\E \left[ \mathcal{A}_\infty f (F) \right] =0$ for all polynomial functions $f$. Then relation $(\ref{eq:com5})$ implies that, by choosing appropriate 
polynomials $f$, we have $\kappa_r(F)=\kappa_r(F_\infty)$ for $r=2,\cdots,d+1$. Now, combining this observation together with relation $(\ref{eq:com5})$, we infer that 
$$\E \left[ F^n \sum_{r=1}^{d+1} a_r \Big( \Gamma_{r-1}(F) - \E[\Gamma_{r-1}(F)] \Big)   \right]=0, \quad n \ge 2.$$
Using e.g. the Malliavin integrations by parts, and similar argument as in the proof of \cite[Proposition 5]{a-p-p}, the latter equation can be turned into a linear recurrent relation between the cumulants of $F$ of order up to $d+1$. Combining this with the knowledge of the 
$d+1$ first cumulants characterise all the cumulants of $F$ and hence the distribution $F$. Indeed, all the distributions in the second Wiener chaos are determined by their moments/cumulants \cite[Proposition 2.7.13, item 3]{n-pe-1}.

\end{proof}

\begin{proof}[Proof of Lemma \ref{sec:four-appr-stein-1}]
$(\Rightarrow)$.   Let us introduce two differential operators characterized by their
  symbols in Fourier domain. For smooth enough test functions, $f$,
  we define:
\begin{align*}
&\mathcal{A}_{d}(f)(x)=\frac{1}{2\pi}\int_{\mathbb{R}}\mathcal{F}(f)(\xi)\bigg(
   A_d(\iu \xi) \bigg)\exp(ix\xi)d\xi,\\ 
&\mathcal{B}_{d'}(f)(x)=\frac{1}{2\pi}\int_{\mathbb{R}}\mathcal{F}(f)(\xi)\bigg(B_{d'}(\iu
  \xi)\bigg)\exp(ix\xi)d\xi,
\end{align*}
with $\mathcal{F}(f)(\xi)=\int_{\mathbb{R}}f(x)\exp(-ix\xi)dx$. 
Integrating against smooth test functions the differential equation
satistifed by the characteristic function $\phi_F$, we have, for the
left hand side: 
\begin{align*}
\int_{\mathbb{R}}\mathcal{F}(\phi)(\xi)A_d(\iu \xi)\dfrac{d}{d\xi}\bigg(\phi_F(\xi)\bigg)d\xi&=\int_{\mathbb{R}}\mathcal{F}\big(\mathcal{A}_{d}(f)\big)(\xi)\dfrac{d}{d\xi}\bigg(\phi_F(\xi)\bigg)d\xi,\\
&=-\int_{\mathbb{R}}\dfrac{d}{d\xi}\bigg(\mathcal{F}\big(\mathcal{A}_{d}(f)\big)(\xi)\bigg)\phi_F(\xi)d\xi,\\ 
&=i\int_{\mathbb{R}}\mathcal{F}\big(x\mathcal{A}_{d}(f)\big)(\xi)\phi_F(\xi)d\xi,
\end{align*}
where we have used the standard fact
$d/d\xi(\mathcal{F}(f)(\xi))=-i\mathcal{F}(xf)(\xi)$. Similarly, for
the right hand side, we obtain: 
\begin{align*}
\operatorname{RHS}&=\int_{\mathbb{R}}\mathcal{F}(f)(\xi)\bigg(B_{d'}(\iu
                    \xi) \bigg)\phi_F(\xi)d\xi
=i\int_{\mathbb{R}}\mathcal{F}\big(\mathcal{B}_{d'}(f)\big)(\xi)\phi_F(\xi)d\xi.
\end{align*}
Thus,
\begin{align*}
\int_{\mathbb{R}}\mathcal{F}\big(x\mathcal{A}_{d}(f)-\mathcal{B}_{d'}(f)\big)(\xi)\phi_F(\xi)d\xi=0
\end{align*}
for all $f \in \mathcal{S}(\R)$.  Going back in the space domain, we
obtain the claim. 

\

$(\Leftarrow)$. We denote $S'(\mathbb{R})$ the space of tempered
distributions. Let $Y$ be a real valued random variable such that
$\E[|Y|]<+\infty$ and
\begin{equation}\label{eq:17}
\forall f\in S(\mathbb{R}),\ \E \bigg[  F\mathcal{A}_{d}f(F)-\mathcal{B}_{d'}f(F) \bigg]=0.
\end{equation}
Since $\E[|Y|]<+\infty$, the characteristic function of $Y$ is
differentiable on the whole real line. Working similarly as in the
first part of the proof (from space domain to Fourier domain),
identity \eqref{eq:17} leads to 
\begin{align*}
  A_d(\iu \cdot) \frac{d}{d \xi}\left( \phi_Y \right)(\cdot) = \iu B_{d'}(\iu \cdot)  \phi_F(\cdot)
\end{align*}
in $S'(\mathbb{R})$.  We also have $\phi_Y(0)=1$ thus, by
Cauchy-Lipschitz theorem, we have:
\begin{align*}
\forall\xi\in\mathbb{R},\ \phi_Y(\xi)=\phi_F(\xi).
\end{align*}
This concludes the proof of the Lemma.
\end{proof}

\begin{proof}[Proof of Theorem \ref{benjarras}]
Let $r_1 = \sum_{k=1}^d\lambda_k{m_k\alpha_k}{c_k}$.   The CF of
random variables as in   \eqref{eq:23} is 
  \begin{equation*}
    \phi_F(\xi) = e^{-i \xi r_1}
    \prod_{j=1}^d \bigg(1-\iu \xi \lambda_j
    c_j\bigg)^{-m_j\alpha_j}.  \end{equation*}
Taking derivatives with respect to $\xi$ one sees that 
\begin{equation*}
  \phi_F'(\xi) = - \iu \left(r_1 + \sum_{j=1}^d
    \frac{\lambda_km_k\alpha_kc_k}{1- \iu \xi \lambda_kc_k}\right) \phi_F(\xi)
\end{equation*}
which, after straightforward simplifications, becomes (we denote
$\nu_j=1/(c_j\lambda_j)$ and $m\alpha = (m_1\alpha_1, \ldots, m_d \alpha_d)$)
\begin{align*}
\prod_{k=1}^d(\nu_k-i\xi)\phi_F'(\xi)=-\iu \bigg\{r_1\prod_{k=1}^d(\nu_k-i\xi)-\sum_{k=1}^dm_k\alpha_k\prod_{l=1,l\ne k}^d(\nu_l-i\xi)\bigg\}\phi_F(\xi).
\end{align*}
All that remains is to compute explicitely the coefficients of the
polynomials on either side of the above, i.e. in Lemma
\ref{sec:four-appr-stein-1}'s $\mathcal{A}_d$ and $\mathcal{B}_{d}$.
First of all, let us consider the following polynomial in
$\mathbb{R}[X]$:
\begin{align*}
P(x)=\prod_{j=1}^d(\nu_j-x)=(-1)^d\prod_{j=1}^d(x-\nu_j).
\end{align*}
We denote by $p_0,...,p_d$ the coefficients of $\prod_{j=1}^d(X-\nu_j)$ in the basis $\{1,X,...,X^d\}$. Vieta formula readily give:
\begin{align*}
\forall k\in\{0,...,d\},\ p_k=(-1)^{d+k}e_{d-k}(\nu_1,...,\nu_d),
\end{align*}
It follows that the Fourier symbol of $\mathcal{A}_{d}$ is given by:
\begin{align*}
\prod_{k=1}^d(\nu_k-i\xi)=P(i\xi)=\sum_{k=0}^d(-1)^ke_{d-k}(\nu_1,...\nu_d)(i\xi)^k.
\end{align*}
Thus, we have, for $f$ smooth enough:
\begin{align*}
\mathcal{A}_{d}(f)(x)=\sum_{k=0}^d(-1)^{k}e_{d-k}(\nu_1,...,\nu_d)f^{(k)}(x).
\end{align*}
Let us proceed similarly for the operator $B_{d,m,\nu}$. We denote by
$P_k$ the following polynomial in $\mathbb{R}[X]$ (for any
$k\in\{1,...,d\}$): 
\begin{align*}
P_k(x)=(-1)^{d-1}\prod_{l=1,l\ne k}^d(x-\nu_l).
\end{align*}
A similar argument provides the following expression:
\begin{align*}
P_k(x)=\sum_{l=0}^{d-1}(-1)^le_{d-1-l}(\underline{\nu}_k)x^l,
\end{align*}
where
$\underline{\nu}_k=(\nu_1,...,\nu_{k-1},\nu_{k+1},...,\nu_d)$. Thus,
the symbol of the differential operator $B_{d}$ is given by:
\begin{align*}
\sum_{k=1}^dm_k\alpha_k\prod_{l=1,l\ne k}^d(\nu_l-i\xi)=\sum_{l=0}^{d-1}(-1)^l\bigg(\sum_{k=1}^dm_k\alpha_ke_{d-1-l}(\underline{\nu}_k)\bigg)(i\xi)^l.
\end{align*}
Thus, we have:
\begin{align*}
  B_{d}(f)(x)=\sum_{l=0}^{d-1}(-1)^l\bigg(\sum_{k=1}^dm_k\alpha_ke_{d-1-l}(\underline{\nu}_k)\bigg)f^{(k)}(x).
\end{align*}
Consequently, we obtain:
\begin{align*}
&\E\bigg[(F+r_1)\sum_{k=0}^d(-1)^{k}e_{d-k}(\nu_1,...,\nu_d)f^{(k)}(F)-\sum_{l=0}^{d-1}(-1)^l\bigg(\sum_{k=1}^dm_k\alpha_ke_{d-1-l}(\underline{\nu}_k)\bigg)f^{(k)}(F)\bigg]=0.
\end{align*}
Finally,  it is easy to see that 
\begin{align*}
\forall k\in\{0,...,d\},\
\bigg(\prod_{j=1}^dc_j\lambda_j  \bigg)e_{k}(\nu_1,...,\nu_d)=e_{d-k}({\lambda_1}{c_1},...,{\lambda_d}{c_d})
\end{align*}
and the conclusion follows.  
\end{proof}

\begin{proof}[Proof of Proposition \ref{same}] In order to lighten the
  notations, we consider the target law represented by
  $F=\sum_{i=1}^{d}\lambda_i (N_i^2-1),$ with $\lambda_j\ne \lambda_i$
  if $i\ne j$ and $\{N_1,...,N_d\}$ is a collection of i.i.d. standard
  normal random variables.  By (\ref{eq:SME}), we have, for any smooth
  functions:
\begin{align*}
\mathcal{A}_\infty f (x):= \sum_{l=2}^{d+1} (b_l - a_{l-1} x ) f^{(d+2-l)}(x) - a_{d+1} x f(x).
\end{align*}
By a re-indexing argument, we have:
\begin{align*}
\mathcal{A}_\infty f (x):= \sum_{k=1}^{d} (b_{d+2-k} - a_{d-k+1} x ) f^{(k)}(x) - a_{d+1} x f(x).
\end{align*}
As a warm up, we start by computing $a_{d+1}$ and $a_{d-k+1}$. We have, by definition:
\begin{align*}
a_{d+1}&=\dfrac{P^{(d+1)}(0)}{(d+1)! 2^d}=\frac{1}{2^d},
\end{align*}
where we have used the definition of the polynomial $P(X)$. Moreover, we have:
\begin{align*}
a_{d-k+1}&=\dfrac{P^{(d+1-k)}(0)}{(d+1-k)! 2^{d-k}}=\dfrac{(-1)^k}{2^{d-k}}e_k(\lambda_1,...,\lambda_d),
\end{align*}
where we have used the fact that $P^{(d+1-k)}(0)$ is equal to $(d+1-k)!$ times the $(d-k)$-th coefficient  of the polynomial $\prod (X-\lambda_j)$. Now, let us compute $b_{d+2-k}$. We have, for $k\in \{1,...,d\}$:
\begin{align*}
b_{d+2-k}&=\sum_{r=d+2-k}^{d+1}\dfrac{a_r}{(r+k-d-1)!}\kappa_{r+k-d}(F_\infty)\\
&=2^{k-d}\sum_{r=d+2-k}^{d+1} \frac{P^{(r)}(0)}{r!}\sum_{j=1}^d \lambda_j^{r+k-d}\\
&=2^{k-d}\sum_{r=d+2-k}^{d+1} (-1)^{d+r-1}e_{d-r+1}\big(\lambda_1,...,\lambda_d\big)\sum_{j=1}^d\lambda_j^{r+k-d}\\
&=(-1)^{k+1}e_{k-1}\big(\lambda_1,...,\lambda_d\big)\sum_{j=1}^d\lambda_j^{2}+...+(-1)e_{1}\big(\lambda_1,...,\lambda_d\big)\sum_{j=1}^d\lambda_j^{k}\\
&\qquad +e_{0}\big(\lambda_1,...,\lambda_d\big)\sum_{j=1}^d\lambda_j^{k+1}.
\end{align*}
Now the trick is to note that $\lambda_j e_{l-1}\big((\underline{\lambda}_j)\big)=e_{l}\big(\lambda_1,...,\lambda_d\big)-e_{l}\big((\underline{\lambda}_j)\big)$. Thus, we have:
\begin{align*}
(-1)e_{1}\big(\lambda_1,...,\lambda_d\big)\sum_{j=1}^d\lambda_j^{k}+e_{0}\big(\lambda_1,...,\lambda_d\big)\sum_{j=1}^d\lambda_j^{k+1}=-\sum_{j=1}^d\lambda_j^{k}e_{1}\big((\underline{\lambda}_j)\big).
\end{align*}
Using the previous equality recursively, we obtain:
\begin{align*}
b_{d+2-k}&=2^{k-d}\bigg[(-1)^{k+1}e_{k-1}\big(\lambda_1,...,\lambda_d\big)\sum_{j=1}^d\lambda_j^{2}+(-1)^k\sum_{j=1}^d\lambda_j^3 e_{k-2}\big((\underline{\lambda}_j)\big)\bigg],\\
&=2^{k-d}(-1)^{k}\bigg[\sum_{j=1}^d\lambda_j^2\bigg(-e_{k-1}\big(\lambda_1,...,\lambda_d\big)+e_{k-1}\big(\lambda_1,...,\lambda_d\big)-e_{k-1}\big((\underline{\lambda}_j)\big)\bigg)\bigg],\\
&=2^{k-d}(-1)^{k+1}\sum_{j=1}^d\lambda_j^2e_{k-1}\big((\underline{\lambda}_j)\big),\\
&=2^{k-d}(-1)^{k+1}\sum_{j=1}^d\lambda_j\bigg(e_{k}\big(\lambda_1,...,\lambda_d\big)-e_{k}\big((\underline{\lambda}_j)\big)\bigg).
\end{align*}
Wrapping everything up together, this ends the proof of the proposition.
\end{proof}


 \section*{Acknowledgments}
 YS gratefully acknowledges support from FNRS under Grant MIS
 F.4539.16. We thank C\'eline Esser and Michel Rigo for fruitful
 discussions. Also,
 EA would like to thank the members of the Institute for
 Mathematical Stochastics at Otto-von-Guericke-Universitat Magdeburg,
 and in particular, Claudia Kirch for their warm hospitality where
 most of this research was carried out, and to Finnish Cultural
 Foundation for financial support.

\addcontentsline{toc}{section}{References}%

\emph{E-mail address}, B. Arras {\tt arrasbenjamin@gmail.com }

\emph{E-mail address}, E. Azmoodeh {\tt ehsan.azmoodeh@rub.de}

\emph{E-mail address}, G. Poly {\tt guillaume.poly@univ-rennes1.fr }

\emph{E-mail address}, Y. Swan  {\tt yswan@ulg.ac.be }

\end{document}